\documentclass[square]{ws-procs975x65}

\newcommand{\lk}{\left(}
\newcommand{\rk}{\right)}

\newcommand{\tr}{\textnormal{Tr}}

\newcommand{\R}{\mathbb{R}}

\newcommand{\Le}{L_d}
\newcommand{\Lz}{\frac 14 L_{d-1}}

\begin{document}

\title{TWO-TERM SPECTRAL ASYMPTOTICS FOR THE DIRICHLET LAPLACIAN ON A BOUNDED DOMAIN}

\author{RUPERT L. FRANK}

\address{Department of Mathematics, Princeton University\\
Washington Road, Princeton, NJ 08544, USA\\
E-mail: rlfrank@math.princeton.edu}

\author{LEANDER GEISINGER}

\address{Fachbereich Mathematik und Physik, Universit\"at Stuttgart\\
Pfaffenwaldring 57, 70569 Stuttgart, Germany\\
E-mail: geisinger@mathematik.uni-stuttgart.de}

\begin{center}
\begin{small}
 \emph{Dedicated to Ari Laptev on the occasion of his 60th birthday.}
\end{small}
\end{center}

\begin{abstract}
Let $-\Delta$ denote the Dirichlet Laplace operator on a bounded open set in $\mathbb{R}^d$. We study the sum of the negative eigenvalues of the operator $-h^2 \Delta - 1$ in the semiclassical limit $h \to 0+$. We give a new proof that yields not only the first term of the asymptotic formula but also the second term involving the surface area of the boundary of the set.
The proof is valid under weak smoothness assumptions on the boundary.
\end{abstract}

\keywords{Dirichlet Laplace operator; Semiclassical limit; Weyl's law.}

\bodymatter

\section{Introduction and main result}
\label{sec:int}


\subsection{Introduction}

Let $\Omega$ be a bounded open set in $\R^d$, $d\geq 2$. We consider the Dirichlet Laplace operator $-\Delta_\Omega$ defined as a self-adjoint operator in $L^2(\Omega)$ generated by the form
$$
(v,-\Delta_\Omega v) \, = \, \int_\Omega |\nabla v(x)|^2 dx
$$
with form domain $H_0^1(\Omega)$. Since $\Omega$ is bounded the embedding of $H_0^1(\Omega)$ into $L^2(\Omega)$ is compact and the spectrum of $-\Delta_\Omega$ is discrete. It consists of a series of positive eigenvalues
$0 < \lambda_1 \leq \lambda_2 \leq \dots$ accumulating at infinity only.

In general, the eigenvalues $\lambda_k$ cannot be calculated explicitly and especially for large $k$ it is difficult to evaluate them numerically. Therefore it is interesting to describe the  asymptotic behavior of $\lambda_k$ as $k \to \infty$. This is equivalent to the asymptotics of the negative eigenvalues of the operator
\hrule
\bigskip
\begin{footnotesize}
\copyright\, 2010 by the authors. This paper may be reproduced, in its 
entirety, for non-commercial purposes.
\end{footnotesize}
\pagebreak
$$
H_\Omega \, = \, -h^2 \Delta_\Omega - 1
$$
in the semiclassical limit $h \to 0+$.

The first general result is due to H. Weyl who studied the counting function 
$$
N_\Omega(h) \, = \, \sharp \{ \lambda_k < h^{-2} \} = \tr\lk H_\Omega \rk_-^0 \, .
$$
In 1912 he showed that the first term of its semiclassical limit is given by the phase-space volume \cite{Weyl12}: For any open bounded set $\Omega \subset \R^d$ the limit
$$
N_\Omega(h) \, = \, C_d \, |\Omega| \, h^{-d} + o(h^{-d}) 
$$
holds as $h \to 0+$, where
$$
C_d \, = \, \frac{1}{(2\pi)^d} \int_{\R^d} (|p|^2 -1)^0_- dp \, = \, \frac{\omega_d}{(2\pi)^d}
$$
and $\omega_d$ denotes the volume of the unit ball in $\R^d$.

H. Weyl conjectured in \cite{Weyl13} that this formula can be refined by a second term of order $h^{-d+1}$ depending on the boundary of $\Omega$. This stimulated a detailed analysis of the semiclassical limit of partial differential operators. We refer to the books \cite{Hoerma85,SafVas97,Ivrii98} for general results and an overview over the literature. 
Eventually, the existence of a second term was proved by V. Ivrii by means of a detailed microlocal analysis \cite{Ivrii80}: If the boundary of $\Omega$ is smooth and if the measure of all periodic geodesic billiards is zero then the limit
\begin{equation}
\label{eq:int:ivrii}
N_\Omega(h) \, = \, C_d \, |\Omega| \, h^{-d} - \frac 14 C_{d-1} \, |\partial \Omega| \, h^{-d+1} + o(h^{-d+1})
\end{equation}
holds as $h \to 0+$, where $|\partial \Omega|$ denotes the surface area of the boundary. 

In this article we are interested in the sum of the negative eigenvalues
$$
\tr (H_\Omega)_- \, = \, \sum (h^2 \lambda_k-1)_- \, .
$$
This quantity describes the energy of non-interacting, fermionic particles trapped in $\Omega$ and plays an important role in physical applications.

The asymptotic relation (\ref{eq:int:ivrii}) immediately implies a refined formula for the semiclassical limit of $\tr (H_\Omega)_-$:
Suppose that the aforementioned geometric conditions on $\Omega$ are satisfied. Then integrating (\ref{eq:int:ivrii}) yields
\begin{equation}
\label{eq:int:trace}
\tr (H_\Omega)_- \, = \, L_d \, |\Omega| \, h^{-d} - \frac 14 L_{d-1} |\partial \Omega| h^{-d+1} + o(h^{-d+1})
\end{equation}
as $h \to 0+$, with 
$$
L_d \, = \, \int_{\R^d} (|p|^2 -1)_- dp \, = \, \frac{2}{d(d+2)}  \frac{\omega_d}{(2\pi)^d} \, .
$$

In the following we present a direct approach to derive the semiclassical limit of $\tr (H_\Omega)_-$. We prove (\ref{eq:int:trace}) without using the result for the counting function. Since we do not apply any microlocal methods the proof works under much weaker conditions.


\subsection{Main Result}

Our main result holds without any global geometric conditions on $\Omega$. We only require weak  smoothness conditions on the boundary - namely that the boundary belongs to the class $C^{1,\alpha}$ for some $\alpha > 0$. That means, we assume that the local charts of $\Omega$ are differentiable and the derivatives are H\"older continuous with exponent $\alpha$.

\begin{theorem}
\label{thm:main}
Let the boundary of $\Omega$ satisfy $\partial \Omega \in C^{1,\alpha}$, $0 < \alpha \leq 1$.
Then the asymptotic limit
$$
\tr (H_\Omega)_- \, = \, \Le \, |\Omega| \, h^{-d} - \Lz \, |\partial \Omega| \, h^{-d+1} + O \lk h^{-d+1+\alpha/(2+\alpha)} \rk
$$
holds as $h \to 0+$.
\end{theorem}

Our work was stimulated by the question whether similar two-term formulae hold for non-local, non-smooth operators. This is unknown, since the microlocal methods leading to (\ref{eq:int:ivrii}) are not applicable. Therefore it is  necessary to use a direct approach.

Indeed, Theorem \ref{thm:main} can be extended to fractional powers of the Dirichlet Laplace operator \cite{FraGei10}. The strategy of the proof is similar but dealing with non-local operators is more difficult and elaborate. In order to give a flavor of our techniques we confine ourselves in this article to the local case.

The question whether the second term of the semiclassical limit of $\tr(H_\Omega)_-$ exists for Lipschitz domains $\Omega$ remains open.

\subsection{Strategy of the proof}

The proof of Theorem \ref{thm:main} is divided into three steps: First, we localize the operator $H_\Omega$ into balls, whose size varies depending on the distance to the complement of $\Omega$. Then we analyze separately the semiclassical limit in the bulk and at the boundary.

To localize, let $d(u) = \inf \{|x-u| \, : \,  \, x \notin \Omega  \}$ denote the distance of $u \in \R^d$ to the complement of $\Omega$. We set 
$$
l(u) \, = \, \frac 12 \lk 1 + \lk d(u)^2 + l_0^2 \rk^{-1/2} \rk^{-1} \, , 
$$
where  $0 < l_0 \leq 1$ is a parameter depending only on $h$. Indeed, we will finally choose $l_0$ proportional to $h^{2/(\alpha+2)}$.

In Section \ref{sec:loc} we introduce real-valued functions $\phi_u \in C_0^\infty(\R^d)$ with support in the ball $B_u = \{ x \in \R^d \, : \, |x-u| < l(u) \}$. For all $u \in \R^d$ these functions satisfy
\begin{equation}
\label{eq:int:grad}
\left\| \phi_u \right\|_\infty \, \leq \, C \ ,  \qquad \left\| \nabla \phi_u \right\|_\infty \leq C \, l(u)^{-1}
\end{equation}
and for all $x \in \R^d$
\begin{equation}
\label{eq:int:unity}
\int_{\R^d} \phi_u^2(x) \, l(u)^{-d} \, du \, = \, 1 \, .
\end{equation}
Here and in the following the letter $C$ denotes various positive constants that might depend on $\Omega$, but that are independent of $u$, $l_0$ and $h$. 

\begin{proposition}
\label{pro:loc}
For $0 < l_0 \leq 1$ and $h > 0$ we have
$$
\left| \textnormal{Tr} (H_\Omega)_- - \int_{\R^d} \textnormal{Tr} \lk \phi_u H_\Omega \phi_u \rk_- l(u)^{-d} \, du \right| \, \leq \, C \, l_0^{-1} \, h^{-d+2} \, .
$$
\end{proposition}

In view of this result, one can analyze the local asymptotics, i. e., the asymptotic behavior of  $\tr(\phi_u H_\Omega \phi_u)_-$ separately on different parts of $\Omega$. First, in the bulk, where the influence of the boundary is not felt.

\begin{proposition}
\label{pro:bulk}
Assume that $\phi \in C_0^\infty(\Omega)$ is supported in a ball of radius $l > 0$ and that 
\begin{equation}
\label{eq:int:gradphi}
\| \nabla \phi \|_\infty \, \leq \, C \, l^{-1}
\end{equation}
is satisfied. Then for $h > 0$ the estimate
$$
\left| \tr \lk \phi H_{\Omega} \phi \rk_- - \Le \int_\Omega \phi^2(x) \, dx \, h^{-d} \right| \, \leq \, C \, l^{d-2} \, h^{-d+2}
$$
holds, with a constant depending only on the constant in (\ref{eq:int:gradphi}).
\end{proposition}

Close to the boundary of $\Omega$, more precisely, if the support of $\phi$ intersects the boundary, a term of order $h^{-d+1}$ appears:

\begin{proposition}
\label{pro:boundary}
Assume that $\phi \in C_0^\infty(\R^d)$ is supported in a ball of radius $l > 0$ intersecting the boundary of $\Omega$ and that inequality (\ref{eq:int:gradphi}) is satisfied.

Then for all $0< l \leq 1$ and $0 < h \leq 1$ the estimate
$$
\left| \tr \lk \phi H_{\Omega} \phi \rk_- - \Le \int_\Omega \phi^2(x) \, dx \, h^{-d} + \Lz \int_{\partial \Omega} \phi^2(x) d\sigma(x) \, h^{-d+1} \right| \, \leq \, r(l,h)
$$
holds. Here $d\sigma$ denotes the $d-1$-dimensional volume element of $\partial \Omega$ and the remainder satisfies
$$
r(l,h) \, \leq \, C \lk \frac{l^{d-2}}{h^{d-2}} + \frac{l^{2\alpha+d-1}}{h^{d-1}} + \frac{l^{d+\alpha}}{h^d} \rk 
$$
with a constant depending  on $\Omega$,  $\| \phi \|_\infty$ and the constant in (\ref{eq:int:gradphi}).
\end{proposition}

Based on these propositions we can complete the proof of the main result. 

\begin{proof}[Proof of Theorem \ref{thm:main}]
In order to apply Proposition \ref{pro:boundary} to the operators $\phi_u H_\Omega \phi_u$, we need to estimate $l(u)$ uniformly.
Assume that $u \in \R^d$ satisfies $B_u \cap \partial \Omega \neq \emptyset$. Then we have $d(u) \leq l(u)$, which by definition of $l(u)$ implies
\begin{equation}
\label{eq:int:luup}
l(u) \leq \,  l_0/ \sqrt3 \, .
\end{equation}
In view of (\ref{eq:int:grad}) we can therefore apply Proposition \ref{pro:bulk} and Proposition \ref{pro:boundary} to all functions $\phi_u$, $u \in \R^d$. Combining these results with Proposition \ref{pro:loc} we get
\begin{align*}
& \left| \tr \lk H_{\Omega} \rk_-  - \frac{\Le}{h^d} \int_{\R^d} \int_\Omega \phi_u^2(x) dx \, \frac{du}{l(u)^d} + \frac{L_{d-1}}{4h^{d-1}} \int_{\R^d} \int_{\partial \Omega} \phi_u^2(x) d\sigma(x) \, \frac{du}{l(u)^d} \right| \\
& \, \leq \, C \lk l_0^{-1} h^{-d+2}  + \int_{U_1} l(u)^{-2} \, du \, h^{-d+2} + \int_{U_2} r(l(u),h) l(u)^{-d} \, du \rk \, ,
\end{align*}
where $U_1 = \{ u \in \Omega \, : \, B_u \cap \partial \Omega = \emptyset \}$ and $U_2 = \{ u \in \R^d \, : \, B_u \cap \partial \Omega \neq \emptyset \}$. Now we change the order of integration and by virtue of (\ref{eq:int:unity}) we obtain
\begin{align}
\nonumber
& \left| \tr \lk H_{\Omega} \rk_-  - \Le \, |\Omega| \, h^{-d} + \Lz \, |\partial \Omega| \, h^{-d+1} \right| \\
\label{eq:int:remainders}
& \, \leq \, C \lk l_0^{-1} h^{-d+2} + \int_{U_1} l(u)^{-2} \, du \, h^{-d+2} + \int_{U_2} r(l(u),h) l(u)^{-d} \, du \rk \, .
\end{align}

It remains to estimate the remainder terms.
Note that, by definition of $l(u)$, we have
$$
l(u) \geq \,  \frac 14 \min \lk d(u), 1 \rk \quad \textnormal{and} \quad l(u) \, \geq \frac{l_0}{4} 
$$
for all $u \in \R^d$. Together with $(\ref{eq:int:luup})$  this implies 
\begin{equation}
\label{eq:int:U}
\int_{U_1} l(u)^{-2} du \, \leq \, C l_0^{-1} \quad \textnormal{and} \quad \int_{U_2} l(u)^a du \, \leq \, C l_0^a \int_{\{ d(u) \leq l_0 \} } du \, \leq \, Cl_0^{a+1} 
\end{equation}
for any $a \in \R$.
Inserting these estimates into (\ref{eq:int:remainders}) we find that the remainder terms are bounded from above by a constant times
$$
l_0^{-1} h^{-d+2} + l_0^{2\alpha} h^{-d+1} + l_0^{\alpha+1} h^{-d}  \, .
$$
Finally, we choose $l_0$ proportional to $h^{2/(\alpha+2)}$ and conclude that all error terms in (\ref{eq:int:remainders}) equal $O(h^{-d+1+\alpha/(2+\alpha)})$ as $h \to 0+$. 
\end{proof}

The remainder of the text is structured as follows. In Section \ref{sec:asympt} we analyze the local asymptotics and outline the proofs of Proposition \ref{pro:bulk} and \ref{pro:boundary}. In Section \ref{sec:loc}, we perform the localization and, in particular, prove Proposition \ref{pro:loc}.


\section{Local asymptotics}
\label{sec:asympt}

To prove the propositions we need the following rough estimate, a variant of the Berezin-Lieb-Li-Yau inequality \cite{Berezi72b,Lieb73,LiYau83}.

\begin{lemma}
\label{lem:berezin}
For any $\phi \in C_0^\infty(\R^d)$ and $h > 0$
$$
\tr \lk \phi H_\Omega \phi \rk_- \, \leq \, \Le \int_{\R^d} \phi^2(x) \, dx \, h^{-d} \, .
$$
\end{lemma}

\begin{proof}
Let us introduce the operator 
$$
H_0 = -h^2 \Delta - 1 \, ,
$$
defined with form domain $H^1(\R^d)$. The variational principle for sums of 
eigenvalues implies $\tr ( \phi H_\Omega \phi )_-  \leq  \tr ( \phi (H_0)_- \phi )_-$. 
Using the Fourier-transform one can derive an explicit expression for the kernel of $(H_0)_-$ and inserting this yields the claim.
\end{proof}

\subsection{Local asymptotics in the bulk}

First we assume $\phi \in C_0^\infty(\Omega)$.  Then we have $\tr \lk \phi H_\Omega \phi \rk_- \, = \, \tr \lk \phi H_0 \phi \rk_-$,
since the form domains of $\phi H_\Omega \phi$ and $\phi H_0 \phi$ coincide. Moreover, by scaling, we can assume $l = 1$. Thus, to prove Proposition \ref{pro:bulk}, it suffices to establish the estimate
$$
\left| \tr \lk \phi H_0 \phi \rk_- - \Le \int_{\R^d} \phi^2(x) \, dx \, h^{-d} \right| \, \leq \, C  h^{-d+2} 
$$
for $h > 0$. 
The lower bound follows immediately from Lemma \ref{lem:berezin}. The upper bound can be derived in the same way as in the proof of Lemma \ref{lem:mod} below. Indeed, by choosing the trial density matrix $\gamma=\chi (H_0)_-^0 \chi$ we find
$$
\tr \lk \phi H_0 \phi \rk_- \, \geq \, \Le \int_{\R^d} \phi^2(x) \, dx - C_d \int_{\R^d} (\nabla \phi)^2(x) \, dx \, h^{-d+2}
$$
and the claim follows.


\subsection{Straightening the boundary}
\label{ssec:str}

Here we transform the operator $H_\Omega$ locally to an operator given on the half-space $\R^d_+ = \{ y \in \R^d \, : \, y_d > 0 \}$. There we define the operator $H^+$ in the same way as $H_\Omega$, with form domain $H_0^1(\R^d_+)$.

Under the conditions of Proposition \ref{pro:boundary} let $B$ denote the open ball of radius $l > 0$, containing the support of $\phi$.
Choose $x_0 \in B \cap \partial \Omega$ and let $\nu_{x_0}$ be the normed inner normal vector at $x_0$. We choose a Cartesian coordinate system such that $x_0 = 0$ and $\nu_{x_0} = (0, \dots, 0, 1)$, and we write $x = (x',x_d) \in \R^{d-1} \times \R$ for $x \in \R^d$.

For sufficiently small $l > 0$ one can introduce new local coordinates near the boundary. Let $D$ denote the projection of $B$ on the hyperplane given by $x_d =0$.
Since the boundary of $\Omega$ is compact and in $C^{1,\alpha}$, there exists a constant $c>0$, such that for $0<l\leq c$ we can find a real function $f \in C^{1,\alpha}$ given on $D$, satisfying
$$
\partial \Omega \cap B \, = \, \left\{ (x',x_d) \, : \, x' \in D , x_d = f(x') \right\} \cap B \, .
$$
The choice of coordinates implies $f(0) = 0$ and $ \nabla f (0) = 0$.
Since $f \in C^{1,\alpha}$ and the boundary of $\Omega$ is compact we can estimate
\begin{equation}
\label{eq:red:fest}
\sup_{x'\in D} |\nabla f(x')| \, \leq \, C \,  l^\alpha \, ,
\end{equation}
with a constant $C >0$ depending only on $\Omega$, in particular independent of $f$.

Now we introduce new local coordinates given by a diffeomorphism $\varphi \, : \, D \times \R \to \R^d$. We set $y_j \, = \, \varphi_j(x) \, = \, x_j$ for $j = 1, \dots, d-1$ and $y_d \, = \, \varphi_d(x) \, = \, x_d - f(x')$.
Note that the determinant of the Jacobian matrix of $\varphi$ equals $1$ and that the inverse of $\varphi$ is defined on $\textnormal{ran} \, \varphi = D \times \R$. 
There we define $\tilde \phi = \phi \circ \varphi^{-1}$ and extend it by zero to $\R^d$, such that $\tilde \phi \in C_0^1(\R^d)$ and $\| \nabla \tilde \phi \|_\infty \leq Cl^{-1}$ holds.

\begin{lemma}
\label{lem:str}
For $0 < l \leq c$ and any $h > 0$ the estimate
\begin{equation}
\label{eq:str:lem1}
\left| \tr ( \phi H_\Omega \phi )_- - \tr ( \tilde \phi H^+ \tilde \phi )_- \right| \, \leq \, C \, l^{d+\alpha} \, h^{-d}
\end{equation}
holds. Moreover, we have 
\begin{equation}
\label{eq:str:lem2}
\int_\Omega \phi^2(x) \, dx  \, = \, \int_{\R^d_+} \tilde \phi^2(y) \, dy
\end{equation}
and 
\begin{equation}
\label{eq:str:lem3}
\left| \int_{\partial \Omega} \phi^2(x) \, d\sigma(x) - \int_{\R^{d-1}} \tilde \phi^2(y',0) \, dy' \right| \, \leq \, C \, l^{d-1+2\alpha} \, .
\end{equation}
\end{lemma}

\begin{proof}
The definition of $\tilde \phi$ and the fact $\textnormal{det} J\varphi = 1$ immediately give (\ref{eq:str:lem2}). Using (\ref{eq:red:fest}) we estimate
$$
\int_{\partial\Omega} \phi^2(x) d\sigma(x)  =  \int_{\R^{d-1}} \tilde \phi^2(y',0) \sqrt{1+|\nabla f|^2 }  dy'  \leq  \int_{\R^{d-1}} \tilde \phi^2(y',0)  dy' +C l^{d-1+2\alpha}
$$
from which (\ref{eq:str:lem3}) follows.

To prove (\ref{eq:str:lem1})
fix $v \in H_0^1(\Omega)$ with support in $\overline{B}$. For $y \in \textnormal{ran} \, \varphi$ put $\tilde v (y) = v \circ \varphi^{-1}(y)$ and extend $\tilde v$ by zero to $\R^d$. Note that $\tilde v$ belongs to $H_0^1(\R^d_+)$.

An explicit calculation shows
$$
\left| (  \tilde v , -\Delta_{\R^d_+} \tilde v ) - ( v, -\Delta_\Omega  v ) \right| \, \leq  C \, l^\alpha \,  (  \tilde v , -\Delta_{\R^d_+} \tilde v ) \, .
$$
Hence, we find
$$
\tr (\phi H_\Omega \phi )_- \, \leq \, \tr ( \tilde \phi  ( -(1-Cl^\alpha ) h^2 \Delta_{\R^d_+} - 1 ) \tilde \phi )_- \, .
$$
Set $\varepsilon = 2C l^\alpha$ and assume $l$ to be sufficiently small, so that $0 < \varepsilon \leq 1/2$ holds. Then
\begin{align*}
\tr (\phi H_\Omega \phi )_- \, & \leq \, \tr (\tilde \phi ( -(1-Cl^\alpha) h^2 \Delta_{\R^d_+} -1 ) \tilde \phi )_-   \\
& \leq \, \tr ( \tilde \phi (  -h^2 \Delta_{\R^d_+} - 1 ) \tilde \phi )_- + \tr ( \tilde \phi (-(\varepsilon - Cl^\alpha ) h^2 \Delta_{\R^d_+} - \varepsilon ) \tilde \phi )_- \\
& \leq \, \tr ( \tilde \phi H^+ \tilde \phi )_- + \varepsilon \, \tr ( \tilde \phi ( - (h^2/2) \Delta_{\R^d_+} -1 ) \tilde \phi )_- \, .
\end{align*}
By Lemma \ref{lem:berezin} we have $\tr ( \tilde \phi ( - (h^2/2) \Delta_{\R^d_+} -1 ) \tilde \phi )_- \leq C l^d h^{-d}$
and we obtain
$$
\tr ( \phi H_\Omega \phi )_- \, \leq \, \tr (\tilde \phi H^+ \tilde \phi )_- + C \, l^{d+\alpha} \, h^{-d} \, .
$$
Finally, by interchanging the roles of $H_\Omega$ and $H^+$, we get an analogous upper bound and the proof of Lemma \ref{lem:str} is complete.
\end{proof}


\subsection{Local asymptotics in half-space}
\label{ssec:mod}

In view of Lemma \ref{lem:str} we can reduce Proposition \ref{pro:boundary} to a statement concerning the operator $H^+$, given on the half-space $\R^d_+$. Indeed, to prove Proposition \ref{pro:boundary}, it suffices to establish the following result.

\begin{lemma}
\label{lem:mod}
Assume that $\phi \in C_0^1(\R^d)$ is supported in a ball of radius $l > 0$  and that (\ref{eq:int:gradphi}) is satisfied. Then for $h > 0$ the estimate
$$
\left| \tr \lk  \phi H^+  \phi \rk_- - \frac{\Le}{h^d} \int_{\R^d_+} \phi^2(x) dx + \frac{L_{d-1}}{4h^{d-1}} \int_{\R^{d-1}} \phi^2(x',0) dx' \right| \, \leq \, C l^{d-2} h^{-d+2} 
$$
holds with a constant depending only on the constant in (\ref{eq:int:gradphi}).
\end{lemma} 

\begin{proof}
On $\R^d_+$ we can rescale $\phi$ and assume $l =1$. In a first step we prove the estimate
\begin{align}
\nonumber
& \left| \textnormal{Tr} \lk \phi H^+ \phi \rk_- - \frac{L_d}{h^d} \int_{\R^d_+} \! \phi^2(x)  dx + \int_{\R^d_+} \! \phi^2(x) \! \int_{\R^d} \! \cos (2\xi_d x_d h^{-1}) (  |\xi|^2 - 1 )_-  \frac{d\xi \, dx}{(2\pi h)^d} \right| \\
\label{eq:half:cos}
& \, \leq \, C \, h^{-d+2} \, .
\end{align}

To derive a lower bound we use the inequality $\tr \lk \phi H^+ \phi \rk_- \leq  \tr \lk \phi (H^+)_- \phi \rk$
and diagonalize the operator $(H^+)_-$, applying the Fourier-transform in the $x'$-coordinates and the sine-transform in the $x_d$-coordinate. This yields
$$
\tr ( \phi H^+ \phi )_-  \, \leq \, \int_{\R^d_+} \phi^2(x) \int_{\R^d} 2 \sin^2 (\xi_d x_d h^{-1}) \lk  |\xi|^2 - 1 \rk_-  \frac{d\xi \, dx}{(2\pi h)^d} 
$$
and the lower bound follows from the identity
\begin{equation}
\label{eq:half:sincos}
2 \sin^2 (\xi_d x_dh^{-1} ) \, = \, 1 - \cos (2\xi_d x_d h^{-1}) \, .
\end{equation}

To prove the upper bound, define the operator $\gamma = \chi (H^+)_-^0 \chi$ with kernel
$$
\gamma(x,y) \, = \, \frac{2}{(2 \pi h)^d} \, \chi(x) \int_{|\xi|<1} e^{i\xi'(x'-y')/h} \, \sin ( \xi_d x_d h^{-1} ) \sin (\xi_d y_d h^{-1}) d\xi \, \chi(y) \, ,
$$
where $\chi$ denotes the characteristic function of an open ball containing the support of $\phi$. Thus, $\gamma$ is a trace-class operator, satisfying $0 \leq \gamma \leq 1$ and by the variational principle it follows that
\begin{align*}
\lefteqn{\tr (\phi H^+ \phi)_-} \\
\, &\geq \, - \tr (\gamma \phi H^+ \phi)\\
\, &= \, - 2 \int_{|\xi|<1} \lk h^2 \| \nabla  e^{i \xi' \cdot /h}  \sin(\xi_d \cdot h^{-1}) \phi  \|^2_{L^2(\R^d_+)} - \left\| \sin(\xi_d \cdot h^{-1}) \phi  \right\|^2_{L^2(\R^d_+)} \rk \frac{d\xi}{(2\pi h)^d} \\
\, &\geq \,  \int_{\R^d} \lk |\xi|^2 -1 \rk_- \int_{\R^d_+} \phi^2(x) \, 2\sin^2 (\xi_d x_d h^{-1})  \frac{dx \, d\xi}{(2\pi h)^d} - C  h^{-d+2} \, .
\end{align*}
In view of (\ref{eq:half:sincos}) this gives an upper bound and we established (\ref{eq:half:cos}).

We proceed to analyzing the term in (\ref{eq:half:cos}) which contains the cosine. We substitute $x_d = th$ and write
\begin{eqnarray}
\nonumber
\lefteqn{ \int_{\R^d_+} \phi^2(x) \int_{\R^d} \cos (2\xi_d x_d h^{-1}) \lk  |\xi|^2 - 1 \rk_-  \frac{d\xi \, dx}{(2\pi h)^d} } \\
\label{eq:half:rem}
& = & \frac{1}{(2\pi)^d} \int_0^\infty \int_{\R^{d-1}} \phi^2(x',th) dx' \int_{\R^d} \cos(2\xi_d t) \lk  |\xi|^2 - 1 \rk_- \, d\xi \, dt \, h^{-d+1} \, .
\end{eqnarray}
Note that 
\begin{equation}
\label{eq:half:second}
\frac{1}{(2\pi)^d}\int_0^\infty \int_{\R^d} \cos(2\xi_d t) \lk  |\xi|^2 - 1 \rk_- \, d\xi \, dt \, = \, \Lz \, .
\end{equation}
Moreover, in \cite[(9.1.20)]{AbrSte72} it is shown that
\begin{align*}
\int_{\R^d} \cos(2\xi_d t) \lk |\xi|^2-1 \rk_- d\xi  =  C \! \int_0^1 \! \cos(2\xi_d t) (1-\xi_d^2)^{(d+1)/2}d\xi_d =   C \frac{J_{d/2+1}(2t)}{t^{d/2+1}} \, ,
\end{align*}
where $J_{d/2+1}$ denotes the Bessel function of the first kind. We remark that $|J_{d/2+1}(2t)|$ is proportional to $t^{d/2+1}$ as $t \to 0+$ and bounded by a constant times $t^{-1/2}$ as $t \to \infty$, see  \cite[(9.1.7) and (9.2.1)]{AbrSte72}. It follows that
\begin{equation}
\label{eq:half:bessel}
\int_0^\infty t \left| \int_{\R^d} \cos(2\xi_d t) \lk  |\xi|^2 - 1 \rk_- d\xi \right| dt \, \leq \, C \int_0^\infty t^{-d/2} |J_{d/2+1}(2t)| \, dt \, \leq \, C \, .
\end{equation}
In view of (\ref{eq:half:rem}), (\ref{eq:half:second}) and (\ref{eq:half:bessel}) we find 
\begin{align*}
&\left|  \int_{\R^d_+} \phi^2(x) \int_{\R^d} \cos (2\xi_d x_d h^{-1}) \lk  |\xi|^2 - 1 \rk_-  \frac{d\xi \, dx}{(2\pi h)^d} - \frac{L_{d-1}}{4h^{d-1}} \int_{\R^{d-1}} \phi^2(x',0) dx' \right| \\
& \leq \, C h^{-d+2} \, .
\end{align*}
Inserting this into (\ref{eq:half:cos}) proves Lemma \ref{lem:mod}.
\end{proof}
Proposition \ref{pro:boundary} is a consequence of Lemma \ref{lem:str} and Lemma \ref{lem:mod}.


\section{Localization}
\label{sec:loc} 

Here we construct the family of localization functions $(\phi_u)_{u \in \R^d}$ and prove Proposition \ref{pro:loc}. The key idea is to choose the localization depending on the distance to the complement of $\Omega$, see \cite[Theorem 17.1.3]{Hoerma85} and \cite{SolSpi03}.

Fix a real-valued function $\phi \in C_0^\infty(\R^d)$ with support in $\{|x| < 1\}$ and $\| \phi \|_2 = 1$. For $u, x \in \R^d$ let $J(x,u)$ be the Jacobian of the map $u \mapsto (x-u)/l(u)$. We define 
$$
\phi_u(x) \, = \, \phi \lk \frac{x-u}{l(u)} \rk \sqrt{J(x,u)} \, l(u)^{d/2} \, ,
$$
such that $\phi_u$ is supported in $\{ x \, : \, |x-u| < l(u) \}$.
According to \cite{SolSpi03}, the functions $\phi_u$ satisfy (\ref{eq:int:grad}) and (\ref{eq:int:unity}) for all $u \in \R^d$. 

To prove the upper bound in Proposition \ref{pro:loc}, put
$$
\gamma \, = \, \int_{\R^d} \phi_u \, \lk \phi_u H_\Omega \phi_u \rk_-^0 \, \phi_u \, l(u)^{-d} \, du \, .
$$
Obviously, $\gamma \geq 0$ holds and in view of (\ref{eq:int:unity}) also $\gamma \leq 1$. The range of $\gamma$ belongs to $H_0^1(\Omega)$ and by the variational principle it follows that 
$$
- \tr (H_\Omega)_- \, \leq \, \tr \, \gamma H_\Omega \, =  \, -  \int_{\R^d} \tr \lk \phi_u H_\Omega \phi_u \rk_- l(u)^{-d} \, du \, .
$$

To prove the lower bound we make use of the IMS-formula 
$$
\frac 12 \lk f, \phi^2 (-\Delta) f \rk + \frac 12 \lk f, -\Delta \phi^2 f \rk \, = \, \lk f, \phi (-\Delta) \phi f \rk - \lk f, f (\nabla \phi)^2 \rk \, ,
$$
valid for $\phi \in C_0^\infty(\R^d)$ and $f \in H_0^1(\Omega)$. Combining this identity with (\ref{eq:int:unity}) yields
\begin{equation}
\label{eq:loc:ims}
(f, -\Delta f) \, = \, \int_{\R^d} \lk \lk f, \phi_u (-\Delta) \phi_u f \rk - \lk f ,  (\nabla \phi_u)^2 f \rk \rk l(u)^{-d} \, du \, .
\end{equation}
Using (\ref{eq:int:grad}) and (\ref{eq:int:unity}) one can show \cite{SolSpi03}
$$
\int_{\R^d} (\nabla \phi_u)^2(x) l(u)^{-d} \, du \, \leq \, C \int_{\R^d} \phi_{u}^2(x) \, l(u)^{-d-2} \, du \, .
$$
We insert this into (\ref{eq:loc:ims}) and deduce
$$
\tr \lk H_\Omega \rk_- \, \leq \, \int_{\Omega^*} \tr \lk \phi_u \lk - h^2 \Delta - 1 - C h^2 l(u)^{-2} \rk \phi_u \rk_- \, l(u)^{-d} \, du \, ,
$$
where $\Omega^* = \{u \in \R^d \, : \, \textnormal{supp} \phi_u \cap \Omega \neq \emptyset \}$.
To estimate the localization error we use Lemma \ref{lem:berezin}. For any $u \in \R$, let $\rho_u$ be another parameter $0< \rho_u \leq 1/2$ and estimate
\begin{eqnarray*}
\lefteqn{ \textnormal{Tr} \lk \phi_u \lk -h^2\Delta -1 - C h^2  l(u)^{-2} \rk \phi_u \rk_- } \\
& \leq & \textnormal{Tr} \lk \phi_u (-h^2 \Delta -1 ) \phi_u \rk_- + C \, \textnormal{Tr} \lk \phi_u \lk -\rho_u  h^2 \Delta - \rho_u- h^2 l(u)^{-2}  \rk \phi_u \rk_- \\
&\leq& \textnormal{Tr} \lk \phi_u  H_\Omega \phi_u \rk_- + C \, l(u)^d (\rho_u h^2)^{-d/2} \lk \rho_u + h^2 l(u)^{-2} \rk^{1+d/2} \, .
\end{eqnarray*}
With $\rho_u$ proportional to $h^2  l(u)^{-2}$
we find
$$
\tr \lk H_\Omega\rk_- \, \leq \, \int_{\Omega^*} \tr \lk \phi_u  H_\Omega \phi_u \rk_- l(u)^{-d} du + C h^{-d+2} \int_{\Omega^*} l(u)^{-2} du \, .
$$
In view of  (\ref{eq:int:U}) the last integral is bounded by a constant times $l_0^{-1}$ and the proof of Proposition \ref{pro:loc} is complete.

\end{document}